\documentclass{amsart}
\usepackage{amstext,amssymb,amsthm,amsopn,newlfont,graphpap,graphics,graphicx,mathrsfs,enumitem}\usepackage{amstext,amssymb,amsthm,amsopn,newlfont,graphpap,graphics,graphicx,mathrsfs,enumitem}
\allowdisplaybreaks
\usepackage[parfill]{parskip}
\usepackage[noadjust]{cite}
\usepackage{epigraph}
\usepackage[colorlinks=true,
            linkcolor=red,
            urlcolor=blue,
            citecolor=magenta]{hyperref}
\usepackage{color}
\usepackage{mathrsfs}
\allowdisplaybreaks
\theoremstyle{plain}
\newtheorem{thm}{Theorem}[section]
\newtheorem*{thm*}{Theorem}

\newtheorem*{prop*}{Proposition}
\newtheorem{cor}{Corollary}[section]
\newtheorem*{cor*}{Corollary}

\newtheorem*{lem*}{Lemma}
\theoremstyle{definition}
\newtheorem{defn}{Definition}[section]
\newtheorem*{defn*}{Definition}
\newtheorem{exmps}{Examples}[section]
\newtheorem*{exmps*}{Examples}
\newtheorem{exmp}[exmps]{Example}
\newtheorem*{exmp*}{Example}

\newtheorem*{exerc*}{Exercise}

\newtheorem{rems}{Remarks}[section]
\newtheorem*{rems*}{Remarks}
\newtheorem{rem}[rems]{Remark}
\newtheorem*{rem*}{Remark}
\newcommand{\N}{{\mathbb N}}
\newcommand{\Z}{{\mathbb Z}}
\newcommand{\R}{{\mathbb R}}
\newcommand{\C}{{\mathbb C}}
\newcommand{\B}[1]{\mathbb{#1}}

\DeclareRobustCommand{\rchi}{{\mathpalette\irchi\relax}}
\newcommand{\irchi}[2]{\raisebox{\depth}{$#1\chi$}}

\DeclareMathOperator{\diam}{diam}

\begin{document}
\title[On metric space expansions]{On metric space expansions}
\author[Marat V. Markin]{Marat V. Markin}
\address{
Department of Mathematics\newline
California State University, Fresno\newline
5245 N. Backer Avenue, M/S PB 108\newline
Fresno, CA 93740-8001
}
\email{mmarkin@csufresno.edu}
\author{Edward S. Sichel}
\email[Corresponding author]{edsichel@mail.fresnostate.edu}
\dedicatory{}
\subjclass[2020]{Primary 54E40, 54E45; Secondary 46T99}
\keywords{Metric space, expansion, compactness, total boundednessF}
\begin{abstract} 
When finding an original proof to a known result describing expansions on compact metric spaces as surjective isometries, we reveal that relaxing the condition of compactness to total boundedness preserves the isometry property and nearly that of surjectivity. 

While a counterexample is found showing that the converse to the above descriptions do not hold, we are able to characterize boundedness in terms of specific expansions we call \textit{anticontractions}.
\end{abstract}
\maketitle

\section[Introduction]{Introduction}

We take a close look at the nature of expansions on certain metric spaces (compact, totally bounded, and bounded), provide a finer classification for such mappings, and use them to characterize boundedness.

When finding an original proof to a known result describing all expansive mappings on compact metric spaces as surjective isometries \cite[Problem $\textrm{X}.5.13^*$]{Dorogovtsev1987}, we reveal that relaxing the condition of compactness to total boundedness still preserves the isometry property and nearly that of surjectivity. 

We provide a counterexample of a not totally bounded metric space, on which the only expansion is the identity mapping, demonstrating that the converse to the above descriptions do not hold.

Various examples for different types of expansions are furnished, in particular the one of a nonsurjective expansion on a totally bounded ``dial set'' in the complex plane which allows to better understand the essence of the latter.

\section[Preliminaries]{Preliminaries}

Here, we outline certain preliminaries essential for the subsequent discourse (for more, see, e.g., \cite{Kaplansky,Markin2027EFA,Muscat2014,Patty,Sutherland}).

\begin{defn}[Sequential Compactness]\ \\
A set $A$ in a metric space $(X,d)$ is called \textit{sequentially compact}, or compact in the \textit{Bolzano-Weierstrass sense}, if every sequence $\left(x_n\right)_{n\in\N}$ of its elements contains a subsequence convergent to an element of $A$.

A metric space $(X,d)$ is said to be sequentially compact
if sequentially compact is the set $X$.
\end{defn}

\begin{rem}
In a metric space setting, the above definition of compactness is equivalent to compactness in the \textit{Heine-Borel sense} defined via open covers (see, e.g., \cite{Markin2027EFA,Patty}). 

\end{rem}

It is convenient for us to use a sequential definition for total boundedness as well (see e.g., \cite{Markin2027EFA,Muscat2014}).

\begin{defn}[Total Boundedness]\ \\
A set $A$ in a metric space $(X,d)$ is called \textit{totally bounded} if every sequence of its elements contains a fundamental (Cauchy) subsequence.

A metric space $(X,d)$ is said to be totally bounded
if totally bounded is the set $X$.
\end{defn}

\begin{defn}[Boundedness]\ \\
A set $A$ in a metric space $(X,d)$ is said to be \textit{bounded} if
\[
\diam(A):=\sup_{x,y\in X}d(x,y)<\infty,
\]
the number $\diam(A)$ being called the \textit{diameter} of $A$

A metric space $(X,d)$ is said to be bounded
if bounded is the set $X$.
\end{defn}

\begin{rem}\label{remcitb}
In a metric space, a (sequentially) compact set is totally bounded and a totally bounded set is bounded but not vice versa (see, e.g., \cite{Markin2027EFA}).
\end{rem}

\section{Expansions}

Now, we introduce and further classify the focal subject of our study, \textit{expansive mappings} (or \textit{expansions}). 

\begin{defn}[Expansive Mapping]\ \\
Let $(X,d)$ be a metric space.
A mapping $T:X\to X$ on $(X,d)$ such that
\[
\forall\, x,y \in X:\ d(Tx,Ty)\geq d(x,y)
\]
is called an \textit{expansive mapping} (or \textit{expansion}).
\end{defn}

It is important for our discourse to introduce a finer classification of expansions.

\begin{defn}[Types of Expansions]\ \\
Let $(X,d)$ be a metric space.
\begin{enumerate}
\item An expansion $T:X \to X$ such that
\[
\forall\, x,y \in X:\ d(Tx,Ty)=d(x,y)
\]
is called an \textit{isometry}, which is the weakest form of expansive mappings.
\item An expansion $T:X \to X$ such that
\[
\exists\, x,y\in X,\ x\neq y:\ d(Tx,Ty)>d(x,y)
\]
we call a \textit{proper expansion}.  
\item An expansion $T:X \to X$ such that
\[
\forall\, x,y \in X,\ x\neq y:\ d(Tx,Ty)>d(x,y)
\]
we call a \textit{strict expansion}.  
\item Finally, an expansion $T:X \to X$ such that
\[
\exists\, E >1\ \forall\, x,y \in X:\ d(Tx,Ty)\ge Ed(x,y)
\]
we call an \textit{anticontraction} with \textit{expansion constant} $E$.
\end{enumerate} 
\end{defn}

\begin{rem}
Clearly, any \textit{anticontraction} is necessarily a \textit{strict expansion}, which in turn is also a \textit{proper expansion}. However, as the following examples demonstrate, the converse statements are not true.
\end{rem}

\begin{exmps}\label{exmpsexp}\
\begin{enumerate}[label={\arabic*.}]
\item On $\C$ with the standard metric, the mapping 
$$ 
g(z):= e^iz,
$$
i.e., the counterclockwise rotation 
by one radian, is an \textit{isometry} which is not a proper expansion.
\item On the space $\ell_{\infty}$ of all real- or complex-termed \textit{bounded} sequences with its standard \textit{supremum metric}
\begin{equation*}
\ell_{\infty}\ni x:=(x_k)_{k\in\N}, y:=(y_k)_{k\in\N} \mapsto d_\infty(x,y):=\sup_{k\in \N} |x_k-y_k|,
\end{equation*} 
the right shift mapping 
$$ 
\ell_{\infty}\ni (x_1,x_2,x_3 \dots )\mapsto
T (x_1,x_2,x_3 \dots ) := (0,x_1,x_2,x_3 \dots ) \in \ell_{\infty}
$$
is also an \textit{isometry} which is not a proper expansion.
\item On $\ell_{\infty}$, the mapping 
$$ 
\ell_{\infty}\ni (x_1,x_2,x_3 \dots )\mapsto
T (x_1,x_2,x_3 \dots ) := (x_1,x_1^2,x_2,x_2^2,\dots ) \in \ell_{\infty}
$$
is a \textit{proper expansion} that is  not strict, since, for $x:=(1,0,0,\dots),y:=(1/2,0,0,\dots)\in \ell_{\infty}$,
\[
d_\infty(Tx,Ty)=3/4>1/2=d_\infty(x,y),
\]
but, for $x:=(1,0,0,\dots),y:=(0,0,0,\dots)\in \ell_{\infty}$,
\[
d_\infty(Tx,Ty)=1=d_\infty(x,y).
\]

\item (Jacob's Ladder Set)  In the space $L_2(0,\infty)$, consider the set of the equivalence classes $\{f_n\}_{n\in\N}$ represented by the functions
$$
f_n(x):=\sqrt{n}\rchi_{[0,1/n]}(x),\ n\in\N,x\in (0,\infty),
$$
($\rchi_\cdot(\cdot)$ is the \textit{characteristic function} of a set), which is a subset of the \textit{unit sphere}
\[
S(0,1):=\left\{f\in L_2(0,\infty)\,\middle|\,d_2(f,0)=\|f\|_2=1 \right\}.
\]

For any $m,n\in\N$ with $n>m$, we have:
\begin{align*}
\qquad
d_2(f_n,f_m)&=
\| f_n-f_m \|_2 = \left[\int_{0}^\infty |f_n(x)-f_m(x)|^2dx \right]^{1/2} \\
& =  \left[\int_{0}^\infty \left|(\sqrt{n}-\sqrt{m})\rchi_{[0,1/n]}(x) -\sqrt{m} \rchi_{(1/n,1/m]}(x)\right|^2dx\right]^{1/2}\\
& =  \left[\int_{0}^{1/n} (\sqrt{n}-\sqrt{m})^2dx +\int_{1/n}^{1/m} {\sqrt{m}}^2dx\right]^{1/2} \\
& =\left[\frac{m-2\sqrt{m}\sqrt{n}+n}{n}+m\left(\frac{1}{m}-\frac{1}{n}\right)\right]^{1/2}= \left[2-2\sqrt{\frac{m}{n}}\right]^{1/2}.
\end{align*}
The map $Tf_n:=f_{kn}$, $n\in\N$, with an arbitrary fixed $k\in\N$ is an \textit{isometry} 
on $\{f_n\}_{n\in\N}$ since, for any $m,n\in\N$ with $n>m$,
\begin{align*}
\qquad
d_2(Tf_n,Tf_m)&=
\|Tf_n-Tf_m\|_2=\|f_{kn}-f_{km}\|_2 =\left[2-2\sqrt{\frac{km}{kn}}\right]^{1/2}\\
&= \left[2-2\sqrt{\frac{m}{n}}\right]^{1/2}= \|f_m-f_n\|_2=d_2(f_n,f_m).
\end{align*}

On the other hand, the map $Sf_n:=f_{n^2}$, $n\in\N$, is a \textit{strict expansion}
on $\{f_n\}_{n\in\N}$ since, for any $m,n\in\N$ with $n>m$,
\begin{align*}
\qquad \qquad
d_2(Sf_n,Sf_m)&=\|Sf_n-Sf_m\|_2 =\|f_{n^2}-f_{m^2}\|_2= \left[2-2\sqrt{\frac{m^2}{n^2}}\right]^{1/2}\\
&=\left[2-2\frac{m}{n}\right]^{1/2} 
>\left[2-2\sqrt{\frac{m}{n}}\right]^{1/2}= \|f_n-f_m\|_2
=d_2(f_n,f_m),
\end{align*}
which is not an anticontraction since
\[
\dfrac{d_2(Sf_{n^2},Sf_n)}{d_2(f_{n^2},f_n)}
=\dfrac{\left[2-\frac{2}{n}\right]^{1/2} }{\left[2-\frac{2}{\sqrt{n}}\right]^{1/2}}\to 1,\ n\to \infty.
\]
\item On $\R$ with the standard metric, the mapping 
$$ f(x)= 2x$$
is an \textit{anticontraction} with \textit{expansion constant} $E=2$. However, the same mapping, when considered on $\R$ equipped with the metric
\[
\R\ni x,y\mapsto \rho(x,y):=\frac{|x-y|}{|x-y|+1},
\]
turning $\R$ into a \textit{bounded} space
(see, e.g., \cite{Markin2027EFA}), is merely a strict expansion, which is not an anticontraction since
\[
\dfrac{\rho(f(x),f(0))}{\rho(x,0)}
=\dfrac{\rho(2x,0)}{\rho(x,0)}
=\dfrac{\frac{|2x|}{|2x|+1} }{\frac{|x|}{|x|+1}}\to 1,\ x\to \infty.
\]
\end{enumerate}
\end{exmps}

\section{Expansions on Compact Metric Spaces}

\begin{thm}[Expansions on Compact Metric Spaces {\cite[Problem $\textrm{X}.5.13^*$]{Dorogovtsev1987}}]\label{ECMS}\ \\
An expansive mapping $T$ on a compact metric space
$(X,d)$ is a surjection, i.e.,
\[
T(X)=X,
\]
and an isometry, i.e.,
\[
\forall\, x,y\in X:\ d(Tx,Ty)=d(x,y).
\] 
\end{thm}

\begin{proof}
For an arbitrary point $x\in X$, and an increasing sequence $\left(n(k)\right)_{k\in\N}$ of natural numbers, consider the sequence 
$$
\left(x_{n(k)}:=T^{n(k)}x\right)_{k\in\N}
$$
in $(X,d)$.

Since the space $(X,d)$ is \textit{compact}, there exists a convergent subsequence $\left(x_{n(k(j))}\right)_{j\in\N}$, which is necessarily \textit{fundamental}. 

\begin{rem}\label{remtbnc1}
Subsequently, we use only the \textit{fundamentality}, and not the \textit{convergence} of the subsequence, and hence, only the total boundedness and not the compactness of the underlying space (Remark \ref{remcitb}).
\end{rem}

By the fundamentality of $\left(x_{n(k(j))}\right)_{j\in\N}$, without loss of generality, we can regard the indices $n(k(j))$, $j\in\N$, chosen sparsely enough so that
$$
d(x_{n(k(j))},x_{2n(k(j))}) \le \frac{1}{j},\ j\in\N.
$$

Since $T$ is an expansion, 
\begin{equation*}
d(x,x_{n(k(j))}) \leq d(T^{n(k(j))}x, T^{n(k(j))}x_{n(k(j))})
= d(x_{n(k(j))},x_{2n(k(j))}) \leq \frac{1}{j},\ j\in\N. 
\end{equation*} 

We thus conclude that
\[
x_{n(k(j))}=T^{n(k(j))}x\to x,\ j\to\infty,
\]
which implies that the range $T(X)$ is \textit{dense} in $(X,d)$,
i.e.,
\[
\overline{T(X)}=X.
\]

Now, let $x, y \in X$ be arbitrary. Then, for the sequence $\left(x_n:=T^nx\right)_{n\in\N}$,
we can, by the above argument, select a subsequence $\left(x_{n(k)}\right)_{k\in\N}$ such that
\[
x_{n(k)} \to x,\ k\to\infty,
\]
and then, in turn, for the sequence $\left(y_{n(k)}:=T^{n(k)}y\right)_{k\in\N}$, we choose a subsequence $\left(y_{n(k(j))}\right)_{j\in\N}$ for which
$$ 
y_{n(k(j))} \to y,\ j\to\infty.
$$ 

Since $\left(x_{n(k(j))}\right)_{j\in\N}$ is a subsequence of $\left(x_{n(k)}\right)_{k\in\N}$, we also have:
$$ 
\lim_{j\to\infty} x_{n(k(j))} = \lim _{k\to \infty} x_{n(k)} = x.
$$

Then, in view of the expansiveness of $T$, for any $j\in\N$,
$$
d(x,y) \leq d(Tx,Ty) \leq d(T^{n(k(j))}x,T^{n(k(j))}y)
=d(x_{n(k(j))},y_{n(k(j))}).
$$

Whence, passing to the limit as $j\to\infty$, by joint continuity of metric, we arrive at
$$
d(x,y) \leq d(Tx,Ty) \le d(x,y),
$$
which implies that
\[
\forall\, x,y\in X:\ d(Tx,Ty) =d(x,y),
\] 
i.e., $T$ is an \textit{isometry}.

\begin{rem}\label{remtbnc2}
Thus far, only the total boundedness and not the compactness of the underlying space has been utilized (Remark \ref{remcitb}).
\end{rem}

Being an isometry, the mapping $T$ is \textit{continuous}, whence, since $X$ is \textit{compact}, we infer that the image $T(X)$ is compact as well, and therefore \textit{closed} in $(X,d)$ (see, e.g., \cite{Markin2027EFA}). 

In view of the \textit{denseness} and the \textit{closedness} of $T(X)$, we conclude
that 
$$ 
T(X)=\overline{T(X)}=X, 
$$
i.e., $T$ is also a \textit{surjection}, as desired, which completes the proof.

\begin{rem}\label{remcntb}
For the \textit{surjectivity} of $T$, the requirement of the \textit{compactness} of the underlying space is essential, as we rely on the fact the continuous image of a compact set is compact. Example \ref{exmpds} demonstrates that this requirement cannot be relaxed even to total boundedness.
\end{rem} 
\end{proof}

\section{Expansions on Totally Bounded Metric Spaces}

We proceed now to demonstrate that relaxing the condition of the compactness of the underlying space to total boundedness yields a slightly weaker result, in which expansions emerge as ``presurjective'' isometries. 

\begin{thm}[Expansions on Totally Bounded  Metric Spaces]\label{ETBMS}\ \\
An expansive mapping $T$ on a totally bounded metric space
$(X,d)$ has a dense range, i.e.,
\[
\overline{T(X)}=X
\]
(``presurjection''), and is an isometry, i.e.,
\[
\forall\, x,y\in X:\ d(Tx,Ty)=d(x,y).
\] 
\end{thm}

\begin{proof}
As is shown in the corresponding part of the proof of Theorem \ref{ECMS} (see Remarks \ref{remtbnc1} and \ref{remtbnc2}), the image $T(X)$ is \textit{dense} in $(X,d)$, i.e.,
\[
\overline{T(X)}=X,
\]
and $T$ is an \textit{isometry}.
\end{proof}

As is mentioned in Remark \ref{remcntb}, the compactness
of the underlying space is essential for the \textit{surjectivity} of expansions, the following  example demonstrating that, when compactness is relaxed to total boundedness, surjectivity is not guaranteed.

\begin{exmp}[Dial Set]\label{exmpds}\ \\
Let 
$$ 
D:= \left\{e^{in}\right\}_{n\in \Z_+}\subset
\left\{z\in \C\,\middle|\,|z|=1 \right\}
$$
($\Z_+$ is the set of \textit{nonnegative integers})
be a \textit{dial set} in the complex plane $\C$ with the usual distance, which is bounded in $\C$, and hence, totally bounded (see, e.g., \cite{Markin2027EFA}), and
\[
D\ni e^{in} \mapsto Te^{in}:=e^{i(n+1)}\in D,\ n\in \Z_+,
\]
be the counterclockwise rotation by one radian,
which is, clearly, an isometry (see Examples \ref{exmpsexp}) but not a surjection on $D$ since, as is easily seen,
$$ 
D \ni 1 = e^{0i}\notin T(D). 
$$
\end{exmp}

\begin{rems}\label{remsds}\
\begin{itemize}
\item This, in particular, implies that, by Theorem \ref{ECMS}, the dial set $D$ is \textit{not compact}, and hence, \textit{not closed}, in $\C$ (see, e.g., \cite{Markin2027EFA}).
\item Thus, on a totally bounded (in particular compact) metric space, no expansion is proper. Any expansion on such a space turns out to be an isometry which, having a dense range, need not be surjective.
\end{itemize}
\end{rems}

By Theorem \ref{ETBMS}, the range $T(D)$ is \textit{dense} in the dial set $D$, which is not closed, relative to the usual distance. This allows us to ``turn the tables'' on the dial set and derive the following rather interesting immediate corollary.

\begin{cor}
Let 
$$ 
D:= \left\{e^{in}\right\}_{n\in \Z_+}.
$$

Then, 
\begin{enumerate}[label={(\arabic*)}]
\item for an arbitrary $n\in \Z_+$, there exists an increasing sequence $\left(n(k)\right)_{k\in\N}$ of natural numbers such that
\[
e^{in(k)}\to e^{in},\ k\to \infty;
\]
\item there exists a $\theta\in \R\setminus \Z_+$ for which there is an increasing sequence $\left(n(k)\right)_{k\in\N}$ of natural numbers such that
\[
e^{in(k)}\to e^{i\theta},\ k\to \infty.
\]
\end{enumerate} 
\end{cor}

\begin{samepage}
\begin{proof}\
\begin{enumerate}
\item Part (1) immediately follows from the fact that, by Theorem \ref{ETBMS}, the range $T(D)=\left\{e^{in}\right\}_{n\in \N}$ is \textit{dense} in $D$.
\item Part (2) follows from the fact that the set $D$, being \textit{not closed} (see Remarks \ref{remsds}), has at least one limit point not belonging to $D$, which, by continuity of metric, is located on the \textit{unit circle} $\left\{z\in \C\,\middle|\,|z|=1 \right\}$, i.e., is of the form $e^{i\theta}$ with some $\theta\in \R\setminus \Z_+$. 
\end{enumerate}
\end{proof}
\end{samepage}

\begin{rem}
If posed as a problem, the prior statement, although simply stated, might be quite challenging to be proved exclusively via the techniques of classical analysis. 
\end{rem}

\section{Are the Converse Statements True?}

Now, there are two natural questions to ask. 

\begin{itemize}
\item If every expansive map $T$ on a metric space $(X,d)$ is a surjective isometry, is the space compact? 
\item If every expansive map $T$ on a metric space $(X,d)$ is a presurjective isometry (see Theorem \ref{ETBMS}), is the space totally bounded? 
\end{itemize}

In other words, do the converse statements to Theorems \ref{ECMS} and \ref{ETBMS} hold? 

The following example answers both questions in the negative.

\begin{exmp}
In the space $\ell_\infty$, consider the bounded set $\{x_n\}_{n\in\B{N}}$ defined by
\[
x_n:=(0,\dots,0,\underbrace{1+\frac{1}{n}}_{\text{$n$th term}},0,\dots ),\ n\in\N,
\]
and let $T$ be an arbitrary expansion on $\{x_n\}_{n\in\B{N}}$. First, we note that, for any expansion, if 
$$\exists\, m,n\in\N,\ m\neq n:\ Tx_m=Tx_{n},$$ 
then
$$ 0 =d(Tx_m,Tx_{n}) < d(x_m,x_{n}) $$
contradicting the expansiveness of $T$. Thus, the mapping $T$ is \textit{injective}.

Observe that 
$$ 
\forall\, m,n=2,3,\ldots:\ d(x_m,x_n)<2=d(x_1,x_n). 
$$ 

Assume 
\begin{equation}\label{asmpt}
Tx_1 \neq x_1.
\end{equation}

Then 
\[
Tx_1 = x_k
\]
with some $k\in \B{N}$, $k\ge 2$. Let $n\in\N$, $n\ge 2$, be arbitrary. 

There are two possibilities: either 
\[
Tx_n\neq x_1
\]
or
\[
Tx_n=x_1.
\]

In the first case, we have:
$$ 
d(Tx_1,Tx_n)= d(x_k,Tx_n) < 2 = d(x_1,x_n). 
$$
\textit{contradicting} the expansiveness of $T$. 

In the second case, for any $m\in \N$, $m\neq n$, by the \textit{injectivity} of $T$,
\[
Tx_{m} \neq x_1,
\]
and hence,
$$ 
d(Tx_1,Tx_{m})= d(x_k,Tx_{m}) < 2 = d(x_1,x_{m}),
$$
which again \textit{contradicts} the expansiveness of $T$.

The obtained contradictions making assumption \eqref{asmpt} false, we conclude that
$$
Tx_1 =x_1.
$$ 

Therefore, by the injectivity of $T$, we can restrict the expansion $T$ to the subset $\{x_n\}_{n\ge 2}$. Applying the same argument, one can show that 
$$
Tx_2=x_2.
$$

Continuing inductively, we see that 
$$
\forall\, n\in\B{N}:\ Tx_n=x_n,
$$
i.e. $T$ is the \textit{identity map}, which is both a \textit{surjection} and an \textit{isometry}, even though the set $\{x_n\}_{n\in\B{N}}$ is not totally bounded, let alone compact (see Remark \ref{remcitb}), as 
$$
\forall\, m,n\in\N,\ m\neq n:\ d_\infty(x_m,x_n)>1.
$$ 
\end{exmp}

\begin{rem}
Thus, a metric space with the property that every expansion on it is a \textit{presurjective isometry} need not be totally bounded. Such spaces, which, by Theorems \ref{ECMS} and \ref{ETBMS}, encompass compact and totally bounded, can be called \textit{nonexpansive}.
\end{rem}

\section{A Characterization of Boundedness}

Although bounded sets support strict expansions (see Examples \ref{exmpsexp} 4, 5). Any attempt to produce an anticontraction on a bounded set would be futile, the following characterization explaining why.

\begin{thm}[Anticontraction Characterization of Boundedness]\ \\
A metric space $(X,d)$ is bounded \textit{iff} no subset of $X$ supports an anticontraction. 
\end{thm}

\begin{proof}
The case of a \textit{singleton} being trivial, suppose that $X$ consists of at least two distinct elements.

\medskip
\textit{``Only if"} part. We proceed \textit{by contradiction}, assuming that $X$ is bounded and there exists a subset $A\subseteq X$ supporting an anticontraction $T: A \to A$ with expansion constant $E$. Then 
$$
\forall\,x,y\in A,x\neq y\ \forall\,n\in\N:\ T^nx,T^ny \in A,
$$ 
which implies
$$ 
\text{diam}(A) \geq d(T^nx,T^ny) \geq E^n d(x,y) \to \infty, \ n \to \infty.  
$$

Hence, $A$ is unbounded, and since $A\subseteq X$, this \textit{contradicts} the boundedness of $X$, the obtained contradiction proving the \textit{``only if''} part.

\textit{``If"} part. Here, we proceed \textit{by contrapositive} assuming $X$ to be \textit{unbounded} and showing that there exists a subset of $X$ which supports an anti-contraction. 

Since $X$ is unbounded, we can select two distinct points $x_1,x_2 \in X$, and subsequently pick $x_3$ so that
$$ \min_{1 \leq i \leq 2} d(x_3,x_i) >2 \max_{1 \leq i,j \leq 2} d(x_i,x_j)  $$ 
Continuing inductively in this fashion, we construct a countably infinite subset $S:=\{x_n\}_{n\in\B{N}}$ of $X$ such that
$$ 
\min_{1 \leq i \leq n} d(x_{n+1},x_i) >2 \max_{1 \leq i,j \leq n} d(x_i,x_j). 
$$

Let
If we then define $T: \{x_n\}_{n\in\B{N}} \to \{x_n\}_{n\in\B{N}} $ by:
$$ 
S\ni x_n\mapsto Tx_n:= x_{n+1}\in S,\ n\in \N. 
$$

Then, for any $m,n\in \N$ with $n>m$,
\begin{align*}
d(Tx_n,Tx_m)&=d(x_{n+1},x_{m+1})\ge \min_{1 \leq i \leq n} d(x_{n+1},x_i)\\
&>2\max_{1 \leq i,j \leq n} d(x_i,x_j) \ge 2 d(x_n,x_m), 
\end{align*} 
which implies that $T$ is an anti-contraction with expansion constant $E=2$ on $S\subseteq X$ completing the proof of the \textit{``if''} part and the entire statement.
\end{proof}

Reformulating equivalently, we arrive at

\begin{thm}[Anticontraction Characterization of Unboundedness]\ \\
A metric space $(X,d)$ is unbounded iff there exists a subset of $X$ which supports an anticontraction. 
\end{thm}


\section{Acknowledgments}

The authors would like to express their appreciation to Drs. Michael Bishop, Przemyslaw Kajetanowicz, and other members of the Functional Analysis and Mathematical Physics Interdepartmental Research Group (FAMP) of California State University, Fresno for insightful questions and stimulating discussions.


\end{document}